\documentclass[12pt]{article}

\usepackage[utf8]{inputenc}
\usepackage[T1]{fontenc}
\usepackage{amsmath,amsthm,amssymb,amsfonts}
\usepackage{mathrsfs}
\usepackage[hidelinks]{hyperref}
\usepackage{geometry}
\usepackage{enumitem}
\numberwithin{equation}{section}

\newtheorem{theorem}{Theorem}[section]

\newtheorem{proposition}[theorem]{Proposition}
\newtheorem{corollary}[theorem]{Corollary}
\newtheorem{definition}[theorem]{Definition}
\newtheorem{remark}[theorem]{Remark}

\newcommand{\SHORTTITLE}[1]{}
\newcommand{\TITLE}[1]{\title{#1}}

\newcommand{\KEYWORDS}[1]{}
\newcommand{\AMSSUBJ}[1]{}
\newcommand{\SUBMITTED}[1]{}
\newcommand{\ACCEPTED}[1]{}
\newcommand{\VOLUME}[1]{}
\newcommand{\YEAR}[1]{}
\newcommand{\PAPERNUM}[1]{}
\newcommand{\DOI}[1]{}

\newcommand{\ABSTRACT}[1]{\begin{abstract}#1\end{abstract}}

\DeclareMathOperator*{\cadlag}{\mathit{D}([0, T], \mathbb{R}^d)}
\DeclareMathOperator*{\R}{\mathbb{R}}

\DeclareMathOperator*{\extension}{a\otimes_s b}
\DeclareMathOperator*{\new_derivative}{\mathit{D}^\gamma F}
\DeclareMathOperator*{\diff_in_dir}{\mathfrak{C}^1_\gamma([0,T],\mathbb{R}^d)}
\DeclareMathOperator*{\partialGamma}{\partial^\gamma F}
\DeclareMathOperator*{\newGradientGamma}{\widetilde{\nabla}^\gamma F}
\DeclareMathOperator*{\diffInDir}{\mathfrak{C}^1_\gamma([0,T],\mathbb{R}^d)}
\DeclareMathOperator*{\diffAllDir}{\mathfrak{C}^1([0,T],\mathbb{R}^d)}
\DeclareMathOperator*{\newGradient}{\widetilde{\nabla} F}
\DeclareMathOperator*{\diffPartition}{\Delta x^n_i}
\begin{document}

\TITLE{Derivatives Along a Curve and the Functional Stochastic Calculus}
\author{Christian Houdr\'e\thanks{School of Mathematics, Georgia Institute of 
    Technology, Atlanta, GA 30332, USA ({\tt houdre@math.gatech.edu}).}
   \thanks{Research supported in part by grants \#524678 and MP-TSM-00002660 from the Simons Foundation.} 
\and
Jorge V\'{i}quez\thanks{This work was conducted while the author was a graduate student at the School of Mathematics, Georgia Institute of Technology, Atlanta, GA, 30332, USA. Additional funding was provided by the University of Costa Rica. ({\tt javiquez42@gmail.com})} }

\date{\today}

\maketitle

\ABSTRACT{Motivated by extending the functional stochastic calculus, 
to important functionals to which it does not apply, a notion of functional derivative along a curve is introduced.  
This new setting is developed by incorporating path-dependent directional extensions. Our results then focus on a comprehensive exploration of these derivatives and the insights  they provide on the structure of functionals.}

\vspace{0.2 cm}
\noindent{\textbf{AMS 2020 subject classifications}: 60G51; 60H05; 60H30}

\vspace{0.1 cm}
\noindent{\textbf{Keywords and phrases}: It\^{o}'s Formula; Functional Stochastic Calculus; Coinvariant Derivative; Malliavin Derivative; Path-dependent Functionals; Feynman–Kac Formula.}

\section{Introduction}

Evaluating the differential effects of extending a path in specific directions has been explored in various deterministic settings, e.g., see \cite{Kim} which was further advanced in \cite{KimBook}. These early works defined the concept of \emph{a coinvariant derivative} as functionals needed to describe the differential effect of extending a continuous path using a Lipschitz function. Further works, such as the one presented in \cite{MANIA} developed the derivative with respect to path extensions introduced in \cite{Chitasvili} to obtain functional integral formulas. However, these approaches are highly dependent on the probability measure of the underlying process, while the results presented below are obtained in a pathwise context, in the spirit of \cite{CONT20101043, Cont13}. 
Notions of vertical and horizontal derivatives for functionals defined on the set of continuous functions are introduced in \cite{Dupire09} where to define the vertical derivative it is first necessary to extend the functionals to the space of c\`{a}dl\`{a}g functions and take the derivative with respect to shocks on the last observed value.  In the following, a notion of space derivative is introduced by computing the differential effect of extending the path along directions determined by a separate functional $\gamma$; recovering, in particular, the horizontal derivative when $\gamma = 0$. Some consequences of the use of these derivatives in the context of the functional It\^{o} calculus are then briefly explored.\\

This article is organized as follows.
Section~\ref{sec:prelim} recalls the necessary background
on the notions of derivative used throughout and establishes the notations
and function spaces needed to state the results precisely.
Some remarks explore further relationships between the $\gamma$-derivative
and previously defined notions, including the horizontal derivative, the
Malliavin derivative, derivatives along smooth paths and the coinvariant
derivative.
Section~\ref{sec:counterexample} then presents an example of a functional
whose $\gamma$-derivatives exist along an infinite-dimensional family of
directions, yet whose horizontal derivative fails to exist at certain paths.  This,
together with the notion of regular directions introduced there, serves
to motivate the developments of Section~\ref{main}.
There, hypotheses are sharpened, leading to the main
Theorem~\ref{existence of horizontal derivative} on the relationship
between the horizontal derivatives and the other directional derivatives.
Next, Theorem~\ref{Coinvariant Ito} studies structural properties of the
directional derivative that lead to an It\^{o} integral formula. The
section then concludes by stating a Fisk-Stratonovich version of the
functional integral formulas, finishing with an application to
path-dependent PDEs via the Feynman--Kac formula.\\

\section{Notations and Definitions}\label{sec:prelim}

\subsection{Definitions and Pathwise Extensions}

This paper borrows the notation of \cite{Dupire09}, \cite{Cont13} and \cite{FirstPaper} to which we also refer for more background material. We work on $C([0,T],\mathbb{R}^d)$, the space of $\mathbb{R}^d$-valued continuous functions defined on $[0,T]$, equipped with the uniform topology. We also work on the Skorokhod space $D([0,T],\mathbb{R}^d)$ equipped 
with any metric $d_D$ weaker than (or the same as) the one induced by the uniform norm, including the metrics associated with each of the Skorokhod topologies.

Throughout, the functionals $F:[0,T]\times D([0,T], \mathbb{R}^d) \to \mathbb{R}$ are assumed to be measurable with respect to $\mathcal{B}([0,T])\otimes \mathcal{F}$, where $\mathcal{B}([0,T])$ is the Borel $\sigma$-field of $[0,T]$. In order to describe the regularity of these functionals, for $w \in D([0, T], \mathbb{R}^d)$, let first $w_{\wedge t}$ denote a path that has been stopped at $t \in [0,T]$, and horizontally extended, i.e., 
\begin{align}
    w_{\wedge t}(s) &:= 
    \begin{cases} 
    w(s), & \text{if } s < t, \\
    w(t), & \text{if } s \geq t.    
    \end{cases}\nonumber
\end{align}
Then $[0,T]\times D([0,T], \mathbb{R}^d)$ is equipped with a pseudo-metric $d_*$:
\begin{align}
    d_*((t,w), (s,v)) := |t-s|+d_D(w_{\wedge t},v_{\wedge s}).\nonumber
\end{align}
Next, a functional $F$ is \emph{non-anticipative} if for all $(t,x) \in [0,T]\times D([0,T],\mathbb{R}^d)$,
    \begin{align}
        F(t,x) &= F(t,x_{\wedge t}).\nonumber
    \end{align}
If $F$ is non-anticipative and $x,y \in D([0,T],\mathbb{R}^d)$ are such that $x_{\wedge t} = y_{\wedge t}$, then $F(t,x) = F(t,y)$. 
\emph{Throughout, all functionals are taken to be non-anticipative.}

\begin{definition}
A functional $F$ is \emph{fixed-time continuous} at $(t,x)$ if for every $\epsilon>0$ there exists $\delta>0$ such that $d_D(x_{\wedge t},y_{\wedge t})<\delta$ implies $|F(t,x_{\wedge t})-F(t,y_{\wedge t})|<\epsilon$, for all $y\in D([0,T],\mathbb{R}^d)$. If this holds for all $(t,x)\in[0,T]\times D([0,T],\mathbb{R}^d)$, $F$ is called fixed-time continuous.
\end{definition}

\begin{definition}
A functional $F$ is \emph{left-continuous} (in time) at $(t,x) \in [0,T]\times D([0,T],\mathbb{R}^d)$ if for all $\epsilon > 0$, there exists $\delta > 0$ such that for all $y \in \cadlag$,
    \begin{align}
        d_*((t,x_{\wedge t}), (s, y_{\wedge s})) <\delta,\, s\leq t \implies |F(t,x_{\wedge t})-F(s,y_{\wedge s})|< \epsilon.\nonumber
    \end{align}
A functional that is left-continuous for all $(t,x)\in [0, T]\times \cadlag$ is said to be left-continuous in time. Right-continuity in time is defined analogously, and a functional is continuous in time if it is both left and right continuous in time.   
\end{definition}

The set of left-continuous functionals is denoted $\mathscr{C}_l([0,T],\mathbb{R}^d)$, while the set of continuous ones 
is denoted $\mathscr{C}([0,T],\mathbb{R}^d)$.

\begin{definition}
A functional $F:[0,T]\times D([0,T],\mathbb{R}^d)\to \mathbb{R}$ is \emph{boundedness-preserving} if for every $t \in [0,T)$ and compact set $K \subset \mathbb{R}^d$, there exists a constant $C_{t,K}> 0$ such that:
        \begin{align}
            x(s) \in K,\,\, \forall s \in [0,t] \implies |F(s,x_{\wedge s})| \leq C_{t,K},\,\,\forall s \in [0,t].\nonumber
        \end{align}
\end{definition}

The set of boundedness-preserving functionals is denoted by $\mathscr{C}_b([0,T],\mathbb{R}^d)$, while $\mathscr{C}^{0,0}_l:=\mathscr{C}_b\cap\mathscr{C}_l([0,T],\mathbb{R}^d)$ and $\mathscr{C}^{0,0}:=\mathscr{C}_b\cap\mathscr{C}([0,T],\mathbb{R}^d)$. 

We first aim to define derivatives along curves.  To do so, we take extensions along an arbitrary direction instead of extending the path by keeping it constant as done with the horizontal derivative. Moreover, to allow dependencies on the process' past behavior, we describe it using a non-anticipative functional $\gamma \in \mathscr{C}^{0,0}([0,T],\mathbb{R}^d)$ which is Lipschitz with respect to its second variable, i.e., requiring that there exists $K > 0$ such that for all $t \in [0,T]$ and $x,y \in \cadlag$,
    \begin{align}\label{path_lipchitz}
        \| \gamma(t,x_{\wedge t})-\gamma(t,y_{\wedge t})\|_2 \leq K\| x_{\wedge t}-y_{\wedge t}\|_\infty,
    \end{align}
    \noindent where $\|\cdot\|_2$ and $\|\cdot\|_\infty$ are, respectively, the Euclidean and the supremum norm.

    \begin{remark}
    When the probability space is augmented to include randomness outside the underlying process, and if the space $(\Omega,\mathcal{F},(\mathcal{F}_t)_{t \in [0, T]},\mathbb{P})$ is such that $\cadlag \subset \Omega$, and $\sigma\{X(s): s\leq t\}\subseteq \mathcal{F}_t$, then one can extend \eqref{path_lipchitz} and take a non-anticipative $\gamma:\Omega \times \cadlag\to \mathbb{R}$ and a random variable $K:\Omega\times [0,T]\to \mathbb{R}$. Moreover, if $T = \infty$, $K$ can also be taken to be an increasing process (see \cite[Section V]{Protter}).
    \end{remark}

The next definition provides notations indicating extensions along a known path after a fixed time $s \in [0,T]$. This will be used to solve path-dependent differential equations by continuously updating the initial condition of the equation. Afterward, it expresses extensions of an underlying process belonging to $\cadlag$ along a specific direction $\gamma$.

    \begin{definition}
        Let $s \in [0,T]$, let $a \in D([0,s],\mathbb{R}^d)$, and let $b \in D([0,T-s],\mathbb{R}^d)$. The function $\extension \in \cadlag$ is defined via
        \begin{align*}
        (a\otimes_s b)(t) := \begin{cases}
        a(t), \text{ $t < s$,}\\
        b(t-s), \text{ $t \geq s$.}
        \end{cases}
        \end{align*}
        If $a, b \in \cadlag$, $a \otimes_s b := a_{[0,s]}\otimes_s b_{[0,T-s]}$, where $a_{[0,s]}$ and $b_{[0,T-s]}$ denote the respective restriction of $a$ and of $b$ to $[0,s]$ and $[0,T-s]$.
    \end{definition}
    
Next, for any $w \in \cadlag$ and $\gamma \in \mathscr{C}^{0,0}([0,T],\mathbb{R}^d)$ satisfying \eqref{path_lipchitz}, the existence and uniqueness of the right-continuous solution to the differential equation
    \begin{align}\label{eq:diffEq}
    \begin{cases}
    dy(t) = \gamma(t,y_{\wedge t})dt,\,\,t > s,\\
    y(t) = w(t), \text{ $t \in [0,s]$},
    \end{cases}
    \end{align}
\noindent is established, following classical ideas such as those presented in \cite[Chapter V, Theorem 7]{Protter}. As defined, $y$ is an absolutely continuous extension of $w$ at time $s$, with a possibly path dependent Radon-Nikod\'{y}m derivative.\\
First assume $s = 0$ and define the operator $I$ acting on $D([0,T],\mathbb{R}^d)$ by
\begin{align*}
    I(x)(t) &:= w(0) + \int_0^{t\wedge \frac{1}{2K}} \gamma(s,x_{\wedge s})\,\mathrm{d}s\\
    &= w(0)+\int_0^{t\wedge \frac{1}{2K}} \gamma(s,0_{\wedge s})\,\mathrm{d}s + \int_0^{t\wedge 1/2K} \left(\gamma(s,x_{\wedge s})-\gamma(s,0_{\wedge s})\right)\,\mathrm{d}s.
\end{align*}
\noindent By definition, $I$ is a contraction with respect to the infinity-norm, and therefore, by the Banach fixed point theorem, there exists a unique continuous solution $x^0$ such that for $t \leq 1/2K$, $x^0(t) = w(0)+\int_0^t \gamma(s,x^0_{\wedge s})\,\mathrm{d}s$. Then, if $s > 0$, a new functional $\gamma_s$ with domain $[s,T]\times D([0,T-s],\mathbb{R}^d)$ can be defined via $\gamma_s(t,x_{\wedge t}) := \gamma(t,(w\otimes_s x)_{\wedge t})$ and the previous observations still allow to define solutions up to, and including, $s+1/2K$. Therefore, using increments of size $1/2K$, a unique continuous solution $x^s$ to
\begin{align}\label{small_step_sol}
\begin{cases}
    x^s(t) = w(s)+\int_s^{t\wedge (s+1/2K)}\gamma_s(l,x^s_{\wedge l})\,\mathrm{d}l, \,\,\text{$t > s$},\\
    x^s(t) = w(t), \,\,\text{$t \leq s$},
\end{cases}
\end{align}
\noindent is obtained. Since this solution is unique, equation \eqref{small_step_sol} allows the recovery of a unique continuous function $x$ satisfying \eqref{eq:diffEq}; this solution to \eqref{eq:diffEq} is denoted by $Y^{s,w,\gamma}$.

The above path extensions $Y^{s,w,\gamma}$, driven by a functional $\gamma$, lead to the definition of a derivative that describes the regularity of a functional along arbitrary directions. For example, this construction allows us to study functionals for which it is only a priori known that they are differentiable in the direction of the process' mean (i.e., $\gamma(t,x_{\wedge t}) = \frac{1}{t}\int_0^t x(s)\mathrm{d}s$), or the process' maximum (i.e., $\gamma(t,x_{\wedge t}) = \sup_{s \in [0,t]} x(s)$).

\subsection{Functional Derivative Along a Curve}

We start this section recalling that the spatial perturbation of $w$ at $t\in [0,T]$ in the direction of $h \in \mathbb{R}^d$, is denoted via:
\begin{align}
    w_{\wedge t}^h(s) &:= 
    \begin{cases} 
    w(s), & \text{if } s < t, \\
    w(t)+h, & \text{if } s \geq t. 
    \end{cases}\nonumber
\end{align}

 We continue by introducing the notion of derivative along a curve.  

\begin{definition}\label{Functional_deriv}
Let $F:[0,T]\times D([0,T],\mathbb{R}^d) \to \mathbb{R}$, and let $\gamma:[0,T]\times D([0,T],\mathbb{R}^d) \to \mathbb{R}^d$. $F$ is said to be differentiable in the direction of $\gamma$ at $(t,x) \in [0,T) \times D([0,T),\mathbb{R}^d)$ if the following limit exists:
\begin{align}\label{eq:Functional_deriv}
\new_derivative(t,x_{\wedge t}) := \lim_{\eta \to 0^+} \frac{F(t+\eta, Y^{t,x,\gamma}_{\wedge (t+\eta)})-F(t,x_{\wedge t})}{\eta},
\end{align}
\noindent where $Y^{t,x,\gamma}$ is the solution to \eqref{eq:diffEq}. Moreover, $F$ is said to be differentiable in the direction $\gamma$ if \eqref{eq:Functional_deriv} exists for every $(t,x)$ in $[0,T) \times D([0,T),\mathbb{R}^d)$.
\end{definition}

When $\gamma = 0$, \eqref{eq:Functional_deriv} reduces to the definition of the 
horizontal (or time) derivative, that is, 
\begin{align}
    DF(t,w) &= \lim_{h\to 0^+} \frac{F(t+h,w_{\wedge t})-F(t,w_{\wedge t})}{h}.
\end{align}

\medskip

The functional It\^{o} calculus, as developed in \cite{Dupire09}, defined its derivatives for functionals evaluated on continuous paths, by first extending them to the space of c\`{a}dl\`{a}g functions. However, as discussed in \cite{ContChiu}, the natural metrics for this space tend to be restrictive, as the uniform metric leads to a non-separable topology, while the $J_1$ topology generated by the Skorokhod metric is such that basic functionals (such as the evaluation functional $(t,x) \to x(t)$) are not continuous. Thus, in order to avoid these issues, and to stay on a natural space for the functionals of continuous paths, the derivative along a curve in Definition~\ref{Functional_deriv} provides a suitable alternative to the time-derivative of \cite{Dupire09}.

\medskip 

\begin{remark}
In Definition~\ref{Functional_deriv}, the Lipschitz property of $\gamma$ guarantees   
the existence and uniqueness of $Y^{s,x,\gamma}$. More generally, if the
Lipschitz condition is relaxed but a process $Y^{s,x,\gamma}$ solving
\eqref{eq:diffEq} still exists for every
$(s,x) \in [0,T) \times D([0,T],\mathbb{R}^d)$, then the notion of derivative
along $\gamma$ can be defined in a similar way.
\end{remark}

\medskip

\begin{remark}
Definition~\ref{Functional_deriv} applies, in particular, when $\gamma$ does not depend on the past $x_{\wedge t}$, allowing for extensions along arbitrary absolutely continuous functions with non path-dependent Radon-Nikod\'{y}m derivative $\gamma$. Moreover, it is also possible to define the limit \eqref{eq:Functional_deriv} by concatenating the original path $x$ with an arbitrary extension $y \in D([0, T-t],\mathbb{R}^d)$. However, being differentiable along all such extensions is too strong of a requirement (take, for example, smooth functionals, with non-vanishing space derivatives, evaluated at paths with non-zero quadratic covariation). Therefore, we deal with derivatives obtained along solutions
to \eqref{eq:Functional_deriv}, as they allow us to represent cases where the functional is differentiable in a specific direction that may depend on the past of the process being evaluated.
\end{remark}

\medskip 

\begin{remark}
The derivative along any fixed smooth path is also obtained from Definition~\ref{Functional_deriv} by taking $\gamma$ to be equal to the derivative of the path. More precisely, if at any $t \in [0,T]$ a derivative is defined using extensions along a fixed, smooth path $y:[0,T] \to \mathbb{R}^d$ with slope $y^\prime(t)$, then one can take $\gamma(t, X_{\wedge t}) = y'(t)$. Moreover, as will be shown in Proposition~\ref{relation_deriv_prop} below, the derivative along $\gamma$ is only influenced by the slope of the extension at any pair $(t,X_{\wedge t})$, one could define the derivative using a functional that extends the path on a fixed slope, regardless of the pair $(t,X_{\wedge t})$. Again, this is covered by Definition \ref{Functional_deriv} by selecting $\gamma$ to be constant. Finally, the ability to extend functions in path-dependent directions will be of interest in the sequel.
\end{remark}

\medskip 

We next recall the notion of space-derivative.

\begin{definition}{\label{Space_Deriv}}
A functional $F$ is said to be space-differentiable at $(t,x) \in [0,T]\times D([0, T],\mathbb{R}^d)$, in the direction of the canonical basis vector $e_i$, $i \in \{1,...,d\}$, if the following limit exists:
\begin{align}
    \partial_i F(t,x_{\wedge t}) &:= \lim_{h \to 0}\frac{F(t,x_{\wedge t}^{he_i})-F(t,x_{\wedge t})}{h}.
\end{align}

If $F$ is space-differentiable at $(t,x)$ in the direction of $e_i$, for every $(t,x) \in [0,T] \times D([0,T],\mathbb{R}^d)$ and every $i \in \{1,\ldots,d\}$, we say that $F$ is \emph{space-differentiable} and denote the gradient of space derivatives by $\nabla F := (\partial_1 F,\ldots,\partial_d F)$.
\end{definition}

\medskip

With the help of these definitions, we finally introduce various function spaces used throughout.  For $i,j$ not both zero, $\mathscr{C}^{i,j}_l([0,T],\mathbb{R}^d)$ is the set of $F:[0,T]\times D([0,T],\mathbb{R}^d)\to\R$ that are $i$-times time-differentiable and $j$-times space-differentiable, with $F$ and all space derivatives in $\mathscr{C}^{0,0}_l$ and time derivatives in $\mathscr{C}_b$ that are fixed-time continuous. Analogously, $\mathscr{C}^{0,0}:=\mathscr{C}_b\cap\mathscr{C}([0,T],\mathbb{R}^d)$, and $\mathscr{C}^{i,j}([0,T],\mathbb{R}^d)$ is defined in a similar way with $\mathscr{C}^{0,0}$ in place of $\mathscr{C}^{0,0}_l$.

\medskip 

The result below, more precisely \eqref{eq:relation_deriv}, indicates that given the existence of the gradient of space derivatives $\nabla F$, the existence of the horizontal derivative $DF$ is equivalent to the existence of $D^\gamma F$, the derivative along the direction $\gamma$.  Its proof follows from the classical functional It\^{o} formula applied to $F$ along the flow $Y^{t,x,\gamma}$, together with a direct application of the dominated convergence theorem.  As this result will be strengthened in Theorem~\ref{existence of horizontal derivative}, its proof is omitted. More generally, if the right-continuity assumption on $DF$, $\nabla F$, and $D^{\gamma}F$ is dropped, Theorem~\ref{existence of horizontal derivative} shows that \eqref{eq:relation_deriv} continues to hold as Radon-Nikod\'{y}m derivatives.

\begin{proposition}\label{relation_deriv_prop}
Let $F \in \mathscr{C}^{1,2}_l([0,T],\mathbb{R}^d)$, and let $\gamma \in \mathscr{C}^{0,0}([0,T],\mathbb{R}^d)$ satisfy \eqref{path_lipchitz}. Suppose that $D^{\gamma}F$ exists and that $DF$, $\nabla F$, and $D^{\gamma}F$ are all right-continuous in time. Then
\begin{align}\label{eq:relation_deriv}
D^\gamma F(t, x_{\wedge t}) = D F(t,x_{\wedge t})+\langle \nabla F(t,x_{\wedge t^-}), \gamma(t,x_{\wedge t})\rangle,
\end{align}
where $\nabla F = (\partial_1 F,\ldots, \partial_d F)$ and $\langle \cdot, \cdot \rangle$ is the Euclidean inner product. 
\end{proposition}

\medskip 

In a converse way, the following corollary shows how the vertical derivatives can be recovered from the path derivatives in this context. It also offers an alternative way of showing that the definition of these derivatives is independent of how a functional is extended to operate on c\`{a}dl\`{a}g functions.  This result follows by applying Proposition~\ref{relation_deriv_prop} to each $\gamma_i$ and inverting the resulting linear system. In particular, the spatial gradient $\nabla F$ can be recovered from the directional derivatives without extending $F$ to operate on c\`{a}dl\`{a}g functions.

\medskip 

\begin{corollary}\label{gradient_recovery}
Let $F \in \mathscr{C}^{1,2}_l([0,T],\mathbb{R}^d)$, and let $\gamma_1,\ldots,\gamma_d \in \mathscr{C}^{0,0}([0,T],\mathbb{R}^d)$ satisfy \eqref{path_lipchitz}, with linearly independent images for all $(t,x) \in [0,T] \times D([0,T],\mathbb{R}^d)$. Assume that $DF, D^{\gamma_1}F,$
$\ldots, D^{\gamma_d}F$ exist and are right-continuous. Define $\Gamma:[0,T]\times D([0,T],\mathbb{R}^d) \to \mathbb{R}^{d\times d}$ with $i$-th row given by $\gamma_i$, and let $D^\gamma F := (D^{\gamma_1}F,\ldots,D^{\gamma_d}F)$. Then
\begin{align*}
\nabla F(t,x_{\wedge t}) = \Gamma^{-1}(t,x_{\wedge t})\bigl(D^\gamma F(t,x_{\wedge t}) - DF(t,x_{\wedge t})\mathbf{1}_d\bigr),
\end{align*}
where $\mathbf{1}_d \in \mathbb{R}^d$ is the vector with all entries equal to $1$. 
\end{corollary}

\begin{remark}\label{Malliavin_remark}
The Malliavin derivative can be seen as the directional derivative along paths in the Cameron-Martin space \cite{Nualart}. In comparison, the derivative along $\gamma$ is defined in a pathwise manner; making no assumption on $\gamma$ being square-integrable since it does not rely on measure invariance along translations by this type of function. Thus, as with the space and time derivatives, the derivative along $\gamma$ can be defined independently of the underlying probability measure.
\end{remark}

\begin{remark}\label{First_Remark}
The multivariate equivalent of the function space used in \cite{Dupire09} can be defined as $\Lambda := \bigcup_{s \in [0,T]}\Lambda_s$, where $\Lambda_s = D([0,s],\mathbb{R}^d)$, while our framework deals with functions defined in $[0,T]\times D([0,T],\mathbb{R}^d)$. Functions from one space to the other can be shown to be equivalent under the identification $u: \Lambda \to [0,T]\times D([0,T],\mathbb{R}^d)$ such that 
\begin{align}\label{exp_notation}
    u(w_t) &= (t,w_{\wedge t}),
\end{align}
i.e., the function $w_t$ with domain $[0,t]$ is mapped to the pair on the right-hand side of \eqref{exp_notation}, which has a path component given by $w_{\wedge t}$. The reverse identification is as follows: for $v:[0,T]\times \cadlag\to \Lambda$, let
\begin{align*}
    v(t,w_{\wedge t}) &:= w_{|[0,t]},
\end{align*}
\noindent i.e. $v$ is the function obtained by restricting the domain of $w$ to the interval $[0,t]$.

Given any $w_t \in \Lambda_t$, note that the differential equation,
\begin{align*}
\begin{cases}
\!\!\quad\frac{dy}{ds}\!\!\!&\!\!=\gamma(s,y_{\wedge s}),\,\text{ $s \in (t,T]$},\\
\quad \!\! y(s) &= w(s),\,\text{ $s \in [0,t]$},
\end{cases}
\end{align*}
\noindent has a unique solution when the function $\gamma$ is Lipschitz. Thus, $y_s := y_{|[0,s]} \in \Lambda_s$ can be defined for all $s \in (t,T]$, and the derivative in the direction of $\gamma$ is given by,
\begin{align*}
    D^\gamma F_t (y_t) = \lim_{h\to 0^+}\frac{F_{t+h}(y_{t+h})-F_t(y_t)}{h}.
\end{align*}
\end{remark}

\bigskip
\bigskip

\section{A Counterexample}\label{sec:counterexample}

The hypothesis $F \in \mathscr{C}^{0,1}_l$ in Theorem \ref{existence of horizontal derivative} is necessary. To see this, define the set of regular directions of $F \in \mathscr{C}^{0,0}_l([0,T],\mathbb{R})$ by
\[
  R(F) := \big\{\gamma \in \mathscr{C}^{0,0}_l([0,T],\mathbb{R}) : D^\gamma F(t,x_{\wedge t}) \text{ exists for all } (t,x) \in [0,T) \times D([0,T],\mathbb{R})\big\}.
\]
If $F \in \mathscr{C}^{0,1}_l$, then Theorem \ref{existence of horizontal derivative} provides the integral representation \eqref{exist_horiz}, from which the horizontal derivative can be recovered. The following shows that this fails without the $\mathscr{C}^{0,1}_l$-hypothesis: for the functional presented below, $R(F)$ is an infinite-dimensional family of directions that does not contain $\gamma = 0$.
 
Let $d=1$, $T=1$, and define the running average $\hat{x}(t) := \frac{1}{t}\int_0^t x(s)\,ds$ for $t > 0$, with $\hat{x}(0) := x(0)$. Set
\begin{equation}\label{eq:counterexample}
  F(t, x_{\wedge t}) := f\!\big(x(t) - 2\hat{x}(t)\big),
\end{equation}
where $f(y) = y\sin(\log|y|)$ for $y \neq 0$ and $f(0) = 0$. Write $\Phi(t,x_{\wedge t}) := x(t) - 2\hat{x}(t)$. Note that $F$ vanishes along every linear path $x(s) = as$, since $\hat{x}(t) = at/2$ and thus $\Phi(t,x_{\wedge t}) = 0$ for all $t > 0$. We show that:
\begin{enumerate}[label=(\roman*)]
  \item\label{item:reg} $F \in \mathscr{C}^{0,0}_l([0,T],\mathbb{R})$,
  \item\label{item:space} the space derivative $\partial_i F$ does not exist at any $(t_0, x)$ with $\Phi(t_0,x_{\wedge t_0}) = 0$ and $t_0 > 0$,
  \item\label{item:horiz} the horizontal derivative $DF$ does not exist at any such $(t_0, x)$ with $\hat{x}(t_0) \neq 0$,
  \item\label{item:R} $R(F) = \big\{\gamma \in \mathscr{C}^{0,0}_l([0,T],\mathbb{R}) : \gamma(t,x_{\wedge t}) = \tfrac{2\hat{x}(t)}{t} \;\text{whenever}\; \Phi(t,x_{\wedge t}) = 0 \;\text{and}\; t > 0 \big\}$.
\end{enumerate}
In particular, $\gamma = 0 \notin R(F)$, since $\gamma = 0$ violates the constraint at any $(t,x)$ with $\Phi(t,x_{\wedge t}) = 0$ and $\hat{x}(t) \neq 0$. Yet $R(F)$ is infinite-dimensional: the constraint only binds on the surface $\{\Phi = 0\}$, leaving $\gamma$ free elsewhere.
 
\medskip
 
\noindent\textit{Proof of\/ \textup{(i)}.}
Since $|f(y)| \leq |y|$ for all $y$ and $|\Phi(t,x_{\wedge t})| \leq |x(t)| + 2|\hat{x}(t)| \leq 3\sup_{s \in [0,t]}|x(s)|$, boundedness-preserving holds: for any compact $K \subset \mathbb{R}$ and $x$ with $x(s) \in K$ for $s \in [0,t]$, one has $|F(s,x_{\wedge s})| \leq 3\sup\{|k| : k \in K\}$ for all $s \in [0,t]$. For fixed-time continuity, both $x \mapsto x(t)$ and $x \mapsto \hat{x}(t) = \frac{1}{t}\int_0^t x(s)\,ds$ are continuous in the uniform topology on $D([0,T],\mathbb{R})$, hence so is $\Phi(t,\cdot)$, and since $f$ is continuous, $F(t,\cdot)$ is continuous. For the joint left-continuity, let $(s_n, y_n) \to (t, x)$ with $s_n \leq t$ in the $d^*$ pseudo-metric. Then $y_n(s_n) \to x(t)$ and
\[
  \hat{x}_{y_n}(s_n) = \frac{1}{s_n}\int_0^{s_n} y_n(u)\,du \;\longrightarrow\; \frac{1}{t}\int_0^t x(u)\,du = \hat{x}(t),
\]
since $\|y_{n\,\wedge s_n} - x_{\wedge t}\|_\infty \to 0$ and $s_n \to t > 0$. Thus $\Phi(s_n, y_{n\,\wedge s_n}) \to \Phi(t, x_{\wedge t})$, and the continuity of $f$ gives $F(s_n, y_{n\,\wedge s_n}) \to F(t, x_{\wedge t})$.
 
\medskip
 
\noindent\textit{Proof of\/ \textup{(ii)}.}
Let $t_0 > 0$ and $\Phi(t_0, x_{\wedge t_0}) = 0$. The spatial perturbation $x^h_{\wedge t_0}$ adds $h$ to $x(t_0)$ while leaving $\hat{x}(t_0)$ unchanged, so $\Phi(t_0, x^h_{\wedge t_0}) = h$. Therefore,
\[
  \frac{F(t_0, x^h_{\wedge t_0}) - F(t_0, x_{\wedge t_0})}{h} = \frac{f(h)}{h} = \sin(\log|h|),
\]
which oscillates as $h \to 0$.
 
\medskip
 
\noindent\textit{Proof of\/ \textup{(iii)} and \textup{(iv)}.}
Fix $t_0 > 0$ and a continuous path $x$. Extend $x$ at time $t_0$ along an admissible direction $\gamma$. Write $\gamma_0 := \gamma(t_0, x_{\wedge t_0})$. By \eqref{eq:diffEq}, the flow satisfies
\[
  Y(t_0 + \eta) = x(t_0) + \int_{t_0}^{t_0+\eta} \gamma(s, Y_{\wedge s})\,ds = x(t_0) + \gamma_0\,\eta + \int_{t_0}^{t_0+\eta}\!\big[\gamma(s,Y_{\wedge s}) - \gamma_0\big]\,ds.
\]
Since $\gamma \in \mathscr{C}^{0,0}_l$, the integrand converges to $0$ as $\eta \to 0^+$, so the last integral is $o(\eta)$. The running average of $Y$ satisfies
\begin{align*}
  (t_0+\eta)\,\hat{x}_Y(t_0+\eta)
  &= \int_0^{t_0} x(s)\,ds + \int_{t_0}^{t_0+\eta} Y(s)\,ds \\
  &= t_0\,\hat{x}(t_0) + x(t_0)\,\eta + \int_{t_0}^{t_0+\eta}\!\int_{t_0}^{s} \gamma(u,Y_{\wedge u})\,du\,ds.
\end{align*}
Since $\gamma$ is boundedness-preserving, there exists $C > 0$ such that $|\gamma(u, Y_{\wedge u})| \leq C$ for $u$ near $t_0$, giving
\[
  \bigg|\int_{t_0}^{t_0+\eta}\!\int_{t_0}^{s}\gamma(u,Y_{\wedge u})\,du\,ds\bigg| \leq \frac{C\eta^2}{2}.
\]
Dividing by $t_0 + \eta$ and rearranging,
\[
  \hat{x}_Y(t_0+\eta) = \hat{x}(t_0) + \frac{x(t_0) - \hat{x}(t_0)}{t_0}\,\eta + o(\eta).
\]
Combining, the auxiliary functional $\Phi$ along the flow evaluates to
\begin{equation}\label{eq:Phi_expansion}
  \Phi(t_0+\eta, Y_{\wedge(t_0+\eta)}) = \Phi(t_0, x_{\wedge t_0}) + \bigg[\gamma_0 - \frac{2\big(x(t_0)-\hat{x}(t_0)\big)}{t_0}\bigg]\,\eta + o(\eta).
\end{equation}
We now consider two cases.
 
\smallskip
 
\noindent\textbf{Case $\Phi(t_0, x_{\wedge t_0}) \neq 0$.} Then $f$ is smooth in a neighbourhood of $\Phi(t_0, x_{\wedge t_0})$. Since~\eqref{eq:Phi_expansion} shows that $\eta \mapsto \Phi(t_0+\eta, Y_{\wedge(t_0+\eta)})$ is differentiable at $\eta = 0$ with derivative $\gamma_0 - 2(x(t_0)-\hat{x}(t_0))/t_0$, the chain rule applied to $f \circ \Phi$ along the flow gives
\[
  D^\gamma F(t_0, x_{\wedge t_0}) = f'\!\big(\Phi(t_0,x_{\wedge t_0})\big)\cdot\bigg[\gamma_0 - \frac{2\big(x(t_0)-\hat{x}(t_0)\big)}{t_0}\bigg],
\]
which exists for every $\gamma$. No constraint on $\gamma$ is imposed.
 
\smallskip
 
\noindent\textbf{Case $\Phi(t_0, x_{\wedge t_0}) = 0$.} Since $x(t_0) = 2\hat{x}(t_0)$, we have $x(t_0) - \hat{x}(t_0) = \hat{x}(t_0)$, and~\eqref{eq:Phi_expansion} reduces to
\[
  \Phi(t_0+\eta, Y_{\wedge(t_0+\eta)}) = \big(\alpha + \varepsilon(\eta)\big)\,\eta,
\]
where $\alpha := \gamma_0 - 2\hat{x}(t_0)/t_0$ and $\varepsilon(\eta) \to 0$ as $\eta \to 0^+$. The difference quotient becomes
\[
  \frac{F(t_0+\eta,Y_{\wedge(t_0+\eta)}) - F(t_0, x_{\wedge t_0})}{\eta} = \frac{f\!\big((\alpha + \varepsilon(\eta))\,\eta\big)}{\eta}.
\]
 
If $\alpha \neq 0$, then for $\eta$ sufficiently small $\alpha + \varepsilon(\eta) \neq 0$, and
\[
  \frac{f\!\big((\alpha+\varepsilon(\eta))\,\eta\big)}{\eta} = \big(\alpha + \varepsilon(\eta)\big)\sin\!\Big(\log\big|\alpha + \varepsilon(\eta)\big| + \log\eta\Big),
\]
whose prefactor converges to $\alpha \neq 0$ while the sine oscillates unboundedly often as $\eta \to 0^+$. The limit does not exist.
 
If $\alpha = 0$, then since $|f(y)| \leq |y|$ for all $y$,
\[
  \bigg|\frac{f\!\big(\varepsilon(\eta)\,\eta\big)}{\eta}\bigg| \leq |\varepsilon(\eta)| \longrightarrow 0,
\]
so $D^\gamma F(t_0, x_{\wedge t_0}) = 0$.
 
\smallskip
 
Therefore, $\gamma \in R(F)$ if and only if the $\alpha = 0$ condition holds at every $(t_0,x)$ with $\Phi(t_0, x_{\wedge t_0}) = 0$ and $t_0 > 0$, establishing~(iv). The horizontal direction $\gamma = 0$ gives $\alpha = -2\hat{x}(t_0)/t_0$, which is nonzero whenever $\hat{x}(t_0) \neq 0$, establishing~(iii).
 
\begin{remark}
  The regular direction $\gamma(t,x_{\wedge t}) = 2\hat{x}(t)/t = \frac{2}{t^2}\int_0^t x(s)\,ds$ is a non-anticipative functional of $x_{\wedge t}$. As is immediate from~\eqref{eq:Phi_expansion} with $\alpha = 0$, this is precisely the direction that preserves the constraint $\Phi = 0$ to first order. Along a linear path $x(s) = as$, it reduces to $\gamma = a$, the slope that continues the path linearly. This confirms that the hypothesis $F \in \mathscr{C}^{0,1}_l$ in Theorem \ref{existence of horizontal derivative} is necessary: without spatial differentiability, $R(F)$ need not contain the horizontal direction, and the structure of the regular directions is governed by a path-dependent constraint on~$\gamma$.
\end{remark}
 
\begin{remark}
  The $\gamma$-derivative can moreover be made right-continuous without restoring the horizontal derivative. Define
  \[
    \gamma^*(t, x_{\wedge t}) := \frac{2\big(x(t) - \hat{x}(t)\big)}{t}, \quad t > 0.
  \]
  For $t > 0$ and continuous paths, $\gamma^*$ is continuous in time (hence right-continuous) and non-anticipative. At any $(t_0, x)$ with $\Phi(t_0, x_{\wedge t_0}) \neq 0$, the function $f$ is smooth near $\Phi(t_0, x_{\wedge t_0})$, and the chain rule gives
  \[
    D^{\gamma^*}\! F(t_0, x_{\wedge t_0}) = f'\!\big(\Phi(t_0, x_{\wedge t_0})\big) \cdot \bigg[\gamma^*_0 - \frac{2\big(x(t_0) - \hat{x}(t_0)\big)}{t_0}\bigg] = 0,
  \]
  since $\gamma^*$ is defined to make the bracketed term vanish identically. At points where $\Phi(t_0, x_{\wedge t_0}) = 0$, the case $\alpha = 0$ in the proof above gives $D^{\gamma^*}\! F(t_0, x_{\wedge t_0}) = 0$ as well. Thus $D^{\gamma^*}\! F \equiv 0$ for all $t > 0$, which is trivially right-continuous. Yet, as shown in~\textup{(iii)}, the horizontal derivative $DF$ still does not exist at points where $\Phi = 0$ and $\hat{x}(t) \neq 0$.
\end{remark}

\section{Integral Formulas and Applications}\label{main}

The forthcoming result presents a relationship between the horizontal and $\gamma$-derivative. The proof strategy emulates the one used in \cite{Dupire09} with the modification that the piecewise constant functions used to approximate the process are replaced in each of the partition's subintervals by the solution to a path-dependent ordinary differential equation.

\begin{theorem}
\label{existence of horizontal derivative}
Let $F \in \mathscr{C}^{0,1}_l([0,T],\mathbb{R}^d)$, be differentiable in the direction of a Lipschitz and boundedness-preserving $\gamma$. Then $F \in \mathscr{C}^{1,1}_l([0,T],\mathbb{R}^d)$, and
\begin{align}\label{horizontal_functional_rel}
DF(t,x_{\wedge t}) &= D^\gamma F(t,x_{\wedge t})-\langle \nabla F(t,x_{\wedge t^-}), \gamma(t,x_{\wedge t})\rangle, \,\,\text{for a.e. $t \in [0,T]$}.
\end{align}
Thus, 
\begin{align}\label{exist_horiz}
    F(t+h,x_{\wedge t})-F(t,x_{\wedge t}) = \int_0^h DF(t+s, x_{\wedge t})\,\mathrm{d}s.
\end{align}
\noindent
\end{theorem}
 In other words, the derivative in the $\gamma$-direction enables the construction of a horizontal derivative in the Radon-Nikod\'{y}m sense. Moreover, if the integrand on the right-hand side of \eqref{exist_horiz} is right-continuous, its integral is differentiable, and this derivative coincides with the horizontal one.
\begin{proof}
First, take $y \in C([0,T],\mathbb{R}^d)$, $y$ of bounded variation with $y(0) = x(0)$. Next, define $y^{k,n}$, $n \geq 1$, $k \in \{1,...,n\}$ sequentially, via:
\begin{align}\label{gamma_extension}
\begin{cases}
    dy^{k,n}(t) = \gamma(t,y^{k,n}_{\wedge t})dt,\\
    y^{k,n}_{\wedge \frac{(k-1)T}{n}^-} = y^{k-1,n}_{\wedge \frac{(k-1)T}{n}},\\
    y^{k,n}(kT/n) = y(kT/n).
\end{cases}
\end{align}

To see that \eqref{gamma_extension} has a unique solution, after defining each $y^{k-1,n}$, one defines $\gamma^{k-1,n}: [0,T-\frac{(k-1)T}{n}]\times D([0,T-\frac{(k-1)T}{n}],\mathbb{R}^d)$ via $\gamma^{k-1,n}(t,w) := \gamma(t,y^{k-1,n}\otimes_{\frac{(k-1)T}{n}} w)$. Thus, \eqref{gamma_extension} turns into:
\begin{align*}
\begin{cases}
dz^{k,n} = \gamma^{(k-1),n}(t,z^{k,n}_{\wedge t})\mathrm{d}t,\\
z^{k,n}(T/n) = y(kT/n),
\end{cases}
\end{align*}
\noindent which has a solution 
using the same Lipschitz arguments as before. Then, let
\begin{align*}
    &y^{k,n} = y^{k-1,n}\otimes_{\frac{(k-1)T}{n}} z^{k,n},\\
    &y^n = \sum_{k = 1}^n y^{k,n}(t){\bf 1}_{[(k-1)T/n,kT/n)}(t).
\end{align*}
\noindent Fix $M > 0$, and let $C$ be such that $\| \gamma(t,z_{\wedge t})\|_2 < C$, for all $z$ satisfying $\| z-y\|_\infty < M$. Note that solutions to \eqref{gamma_extension} have the form $y^{k,n}(t) = y(kT/n)-\int_t^{kT/n}\gamma^{k-1,n}(s,y^{k,n}_{\wedge s})\,\mathrm{d}s = y(kT/n)-\int_t^{kT/n}\gamma(s,y_{\wedge s})\mathrm{d}s-\int_t^{kT/n}(\gamma^{k-1,n}(s,y^{k,n}_{\wedge s})-\gamma(s,y_{\wedge s}))\,\mathrm{d}s$. Moreover, since $y$ is uniformly continuous, there exists $N_1 \geq 1$ such that for $n \geq N_1$, if $|t-s| < 1/n$, then $\vert y(s)-y(t)\vert < \epsilon$. Furthermore, for any $\epsilon > 0$, there also exists $N_2 \geq 1$ such that $KT/N_2 < \epsilon$, where $K$ is from \eqref{path_lipchitz} and $CT/N_2 < M$. Take $n \geq N = N_1 \lor N_2$, and note that for $t <T/n$,
\begin{align*}
    \| y^{1,n}(t)-y(t)\|_2 \leq \epsilon+\frac{CT}{n}+\| y^{1,n}-y\|_\infty \epsilon,
\end{align*}
\noindent implies that,
\begin{align*}
    \| y^{1,n}-y\|_\infty(1-\epsilon) \leq \epsilon+\frac{CT}{n},
\end{align*}

\noindent and that for $\epsilon$ small enough this implies $\| y^{1,n}-y\|_\infty < M$. Assume, for the induction hypothesis, that this is true for $k =1,...,m$, then for $m + 1$, $K$ as in \eqref{path_lipchitz}, and $t \in (mT/n,(m+1)T/n]$,
\begin{align*}
   \| y^{m+1, n}(t)-y(t)\|_2 &\leq \epsilon+\frac{CT}{n}+\left(M \lor  \sup_{t \in (\frac{mT}{n},\frac{(m+1)T}{n}]}\left(\| y^{m+1, n}(t)-y(t)\|_2\right)\right)\epsilon.
\end{align*}
Regardless of which of the two values the maximum takes, by choosing $\epsilon$ small enough, we obtain $\| y^{m+1, n}-y\|_\infty < M$, which in turn implies $\|\gamma^{m,n}(t,y^n_{\wedge t})\|_2 < C$. This common bound can then be used to see that $\| y-y^n\|_\infty \to 0$, uniformly in $n$.

\noindent Finally, define $f^{k,n}:\mathbb{R}^d \to \mathbb{R}$, via $f^{k,n}(z) := F\left(kT/n,(y^n_{\wedge kT/n^-})^{z-y(kT/n)}\right)$, and note that
\begin{align*}
    \nabla f^{k,n}(z) = \nabla F\left(kT/n,(y^n_{\wedge kT/n^-})^{z-y(kT/n)}\right).
\end{align*}
Next, if $\tilde{z}^{k,n}:[Tk/n,T(k+1)/n] \to \mathbb{R}^d$, is defined via,
\begin{align*}
    \tilde{z}^{k,n} := y(t)-\int_{Tk/n}^t \gamma(s,y^n_{\wedge s})\mathrm{d}s.
\end{align*}
\noindent Then,
\begin{align*}
    F(kT/n&,y^n_{\wedge \frac{kT}{n}})-F((k-1)T/n,y^n_{\wedge (k-1)T/n})\\
    =  &F(kT/n,y^n_{\wedge \frac{kT}{n}^-})-F((k-1)T/n,y^n_{\wedge (k-1)T/n})\\
    &+f^{k,n}\left(y((k+1)T/n)-\int_{Tk/n}^{T(k+1)/n} \gamma(t,y^n_{\wedge t})\,\mathrm{d}t\right)-f^{k,n}(y(kT/n))\\
    = &\int_{(k-1)T/n}^{kT/n} D^\gamma F(t,y^n_{\wedge t})\,\mathrm{d}t
    +\int_{kT/n}^{(k+1)T/n} \langle\nabla F(kT/n,(y^n_{\wedge kT/n^-})^{\tilde{z}^{k,n}(t)-y(kT/n)}),\mathrm{d}y(t)\rangle\\
    &-\int_{kT/n}^{(k+1)T/n} \langle \nabla F(kT/n,(y^n_{\wedge kT/n^-})^{\tilde{z}^{k,n}(t)-y(kT/n)}),\gamma(t,y^n_{\wedge t})\rangle\,\mathrm{d}t.
\end{align*}

\noindent Since $F(T,y^n_{\wedge T}) \to F(T,y_{\wedge T})$, and
\begin{align*}
F(T,y^n_{\wedge T})-F(0,y^n_{\wedge 0}) = \sum_{k = 1}^n  \left(F(kT/n,y^n_{\wedge \frac{kT}{n}})-F((k-1)T/n,y^n_{\wedge (k-1)T/n})\right),
\end{align*}

\noindent it follows that:
\begin{align*}
    F(T&,y^n_{\wedge T})-F(T/n,y^n_{\wedge T/n}) \nonumber\\
    = &\int_{T/n}^T  D^\gamma F(t,y^n_{\wedge t})\,\mathrm{d}t\nonumber\\
    &+\int_{2T/n}^T \sum_{k = 2}^{n-1} \langle\nabla F(kT/n,y_{\wedge kT/n^-}^{\tilde{z}^{k,n}(t)-y(kT/n)}),\,\mathrm{d}y(t)\rangle{\bf 1}_{[kT/n,(k+1)T/n)}(t)\nonumber\\
    &-\int_{2T/n}^T \sum_{k = 2}^{n-1} \langle\nabla F(kT/n,y_{\wedge kT/n^-}^{\tilde{z}^{k,n}(t)-y(kT/n)}),\gamma(t,y^{k+1,n}_{\wedge t})\rangle{\bf 1}_{[kT/n,(k+1)T/n)}(t)\,\mathrm{d}t.
\end{align*}

Since $y^{k,n}(t) \to y(t)$ uniformly in $t$, since $\gamma$ and $D^\gamma F$ are fixed-time continuous, since the processes $(\gamma(t,y^n_{\wedge t}))_{ t\in [0,T]}$ share an upper bound uniform in $n\in\mathbb{N}$, and since the space derivatives are left-continuous, by taking the limit as $n \to \infty$, this last equality turns into:
\begin{align*}
 F(T,y_{\wedge T})\!-\!\!F(0,y_{\wedge 0}) = \int_{0}^T \!\!\!\!\!\! D^\gamma F(t,y_{\wedge t})\,\mathrm{d}t+\int_0^T\!\!\!\!\!\! \nabla F(t,y_{\wedge t})\,\mathrm{d}y(t)-\int_0^T\!\!\!\!\! \langle\nabla F(t,y_{\wedge t}),\gamma(t,y_{\wedge t})\rangle\,\mathrm{d}t&.
\end{align*}

If the time $s$ at which the function is to be extended is different from $t = 0$, it suffices to redefine $\bar{y} = x\otimes_s y$, and $F_s:[0,T-s]\times D([0,T-s],\mathbb{R}^d)$ such that $F_s(t,w) = F(t+s, x\otimes_s w)$ and the conclusion of the theorem continues to hold. Thus if $F \in \mathscr{C}^{0,1}_l([0,T],\mathbb{R}^d)$, and there exists a fixed-time continuous $\gamma$-derivative $D^\gamma F$, the following relationship is obtained:
\begin{align}
\label{eq:end_proof1}
F(t+h,x_{\wedge t})\!\!-\!\!F(t,x_{\wedge t})= \int_t^{t+h}\!\!\!\!\!\!D^\gamma F(s,x_{\wedge t})\mathrm{d}s-\int_t^{t+h}\!\!\!\!\!\!\langle \nabla F(s,x_{\wedge t}), \gamma(s,x_{\wedge t})\rangle\mathrm{d}s.
\end{align}
\end{proof}
The representation \eqref{eq:end_proof1} enables us to write extensions along fixed paths as absolutely continuous functions with respect to the Lebesgue measure, which is the property used in the proof of the functional It\^{o} formula. Moreover, if $D^\gamma F, \nabla F$, and $\gamma$ are all right-continuous, then the horizontal derivative exists and is given by \eqref{horizontal_functional_rel}.\\

In Theorem~\ref{existence of horizontal derivative}, there is no reason to prefer the horizontal direction over the others. However, another question that might arise is whether the existence of derivatives along \emph{any} direction allows for an integral formula, enabling to replace the space derivative by this directional derivative. In order to study this question, the following definition is presented.

\begin{definition}\label{notation_derivs}
Let $\gamma \in \mathscr{C}_l^{0,0}([0,T],\mathbb{R}^d)$, and let $\diff_in_dir$ denote the set of all\\
$F \in \mathscr{C}_l^{0,0}([0,T],\mathbb{R}^d)$ such that there exist functionals $\partial^\gamma F, \widetilde{\nabla}^\gamma F$ in $\mathscr{C}^{0,0}([0,T],\mathbb{R}^d)$, for which
\begin{align}\label{diff_direction}
    D^\gamma F(t,x_{\wedge t}) &= \partialGamma(t,x_{\wedge t})+\langle \newGradientGamma(t,x_{\wedge t}),\gamma(t,x_{\wedge t})\rangle,
\end{align}
\noindent for all $(t,x)\in [0,T]\times \cadlag$, with $\widetilde{\nabla}^\gamma F = (\tilde{\partial}^\gamma_1 F, \tilde{\partial}_2^\gamma F,...,\tilde{\partial}_d^\gamma F)$. 
\end{definition}

\medskip

Let now $\diffAllDir$ be the set of all $ F \in \mathscr{C}^{0,0}_l([0,T],\mathbb{R}^d)$ such that for all Lipschitz $\gamma \in \mathscr{C}_l^{0,0}([0,T],\mathbb{R}^d)$, $F \in \diffInDir$ and the functionals $\partial^\gamma F$ and $\tilde{\nabla}^\gamma F$ in \eqref{diff_direction} are the same for all $\gamma$.  (In this instance  they will be denoted respectively as $\partial F$ and $\newGradient$.)  Observe that if $F \in \mathscr{C}^{1,1}([0,T], \mathbb{R}^d)$, then \eqref{diff_direction} applies to all relevant $\gamma$, and thus $\mathscr{C}^{1,1}([0,T],\mathbb{R}^d) \subseteq \diffAllDir$.

\begin{remark}
    If $F \in \diffAllDir$, then $F \in \mathscr{C}^{1,0}([0,T],\mathbb{R}^d)$ and $\partial F = DF$. Moreover, if $F \in \mathscr{C}^{1,1}([0,T],\mathbb{R}^d)$, then as a consequence of Theorem \ref{existence of horizontal derivative}'s proof, $F \in \diffAllDir$, with $\partial F = DF$ and $\tilde{\nabla}F = \nabla F$.
\end{remark}

\medskip 

Definition~\ref{notation_derivs} matches the definition of the coinvariant derivative with respect to absolutely continuous extensions described in \cite{Kim} and recovers the first-order term in the derivatives used in  \cite{Zhang}. Moreover, for $F \in \mathfrak{C}^1([0,T],\mathbb{R}^d)$ and $i > 1$, we say that $F \in \mathfrak{C}^i([0,T],\mathbb{R}^d)$ if $F$ together with $\tilde {\partial}_tF,\tilde{\partial}_1F,...,\tilde{\partial}_dF$ belong to $\mathfrak{C}^{i-1}([0,T],\mathbb{R}^d)$. With this last piece of notation, and using the notion of quadratic covariation along a partition $\pi$, denoted by $[\, \cdot \,]_\pi$, and the integral with respect to a partition, see, e.g. \cite{CONT20101043}, the following theorem holds true:

\medskip 

\begin{theorem}
\label{Coinvariant Ito}
Let $F \in \mathfrak{C}^2([0, T],\mathbb{R}^d)$, let $x \in C([0,T],\mathbb{R}^d)$, let $\pi = \{\pi_n\}_{n \in \mathbb{N}}$ be a sequence of partitions such that $\lim_{n \to \infty} |\pi_n| = 0$, and let $x$ be of bounded quadratic covariation along the partition $\pi$. Then,
\begin{align}\label{Ito_along_a_direction}
    F(t,x_{\wedge t}) -F(0,x_{\wedge 0}) = \int_0^t \partial F(s,x_{\wedge s})\,\mathrm{d}s +\int_0^t \langle\newGradient(s,x_{\wedge s}), \mathrm{d}^\pi x(s)\rangle&\nonumber\\
    + \frac{1}{2}\int_0^t Tr((\tilde{\nabla
    })^2 F(s,x_{\wedge s})\mathrm{d}[x]_\pi(s)).&
\end{align}
\end{theorem}
\begin{proof}
To prove \eqref{Ito_along_a_direction}, an approach similar to the functional It\^{o} formula's  proof presented in \cite{Dupire09} is used, with the modification that a linear approximation is used instead the step-wise constant one. To that extent, take a sequence $\pi_n = \{\tau^n\}$ of partitions of $[0,T]$, such that $\lim_{n \to \infty} |\pi_n| = 0$. Once we obtain the sequence of partitions $\pi = \{\pi_n\}_{n \in \mathbb{N}}$, define $\Delta \tau^n_i = \tau^n_{i+1}-\tau^n_i$ and $\diffPartition = x(\tau^n_{i+1})-x(\tau^n_i)$. Using these terms, define the polygonal approximation $x^n$ of $x$ as
\begin{align*}
    x^n(t) &= \sum_{i = 0}^{k_n-1} \left(x(\tau^n_i)+(t-\tau^n_i)\frac{\Delta x^n_i}{\Delta \tau^n_i}\right){\bf 1}_{[\tau^n_i,\tau^n_{i+1})}(t)+x(T){\bf 1}_{\{T\}}(t).
\end{align*}
\noindent Since $F$ is continuous for the topology of uniform convergence, $F(T,x^n_{\wedge T}) \to F(T,x_{\wedge T})$ as $n \to \infty$. Moreover, for $t \in [\tau^n_i,\tau^n_{i+1})$, $x^n$ behaves as the extension of $x^n$ in the direction of $\Delta x^n_i/\Delta \tau^n_i$ at time $\tau_i^n$, thus the map $\eta \mapsto F(\tau^n_i+\eta, x^n_{\wedge(\tau^n_i+\eta)})$ is twice continuously differentiable by the definition of $\mathfrak{C}^2([0,T],\mathbb{R}^d)$, and a standard Taylor expansion gives, 
\begin{align}
\label{eqn:Ito Error}
F(\tau^n_{i+1},x^n_{\wedge \tau^n_{i+1}})-F(\tau^n_i,x^n_{\wedge \tau^n_i}) = \partial F(\tau^n_i,x^n_{\wedge\tau^n_i})\Delta\tau^n_i+\left\langle \tilde{\nabla} F(\tau^n_i,x^n_{\wedge \tau^n_i}),\Delta x^n_i\right\rangle\nonumber&\\
+\frac{(\Delta \tau^n_i)^2}{2}\partial\partial F(^*\tau_i^n,x^n_{\wedge*\tau^n_i})+\frac{\Delta \tau^n_i}{2}\langle\partial\tilde{\nabla}F(^*\tau_i^n,x^n_{\wedge*\tau^n_i}),\Delta x^n_i\rangle\nonumber &\\
+\frac{1}{2}\langle\Delta x^n_i,(\tilde{\nabla})^2F(\tau^n_i,x^n_{\wedge \tau^n_i}) \Delta x^n_i\rangle+R(\tau^n_i,\tau^n_{i+1},x^n_{\wedge \tau^n_{i+1}})&, 
\end{align}
\noindent where $\partial \tilde{\nabla} F = (\partial\tilde{\partial}_1F,...,\partial\tilde{\partial}_d F)$, and $^*\tau^n_i \in [\tau^n_i,\tau^n_{i+1})$ for all $i \in \{0,1,...,k_n-1\}$. Since $x$ is continuous, $\|x\|_\infty$ is bounded by a constant $M > 0$, and so the residual term $R$ can be bounded in the following manner:
\begin{align*}
    R(\tau^n_i,\tau^n_{i+1},x^n_{\wedge \tau^n_{i+1}}) &\leq \frac{1}{2}\langle\Delta x^n_i,((\tilde{\nabla})^2F(s^n_i,x^n_{\wedge s^n_i})-(\tilde{\nabla})^2 F(\tau^n_i,x^n_{\wedge \tau^n_i}))\Delta x^n_i\rangle,\nonumber\\
    &\leq r((\|x_{\wedge \tau^n_i}-x^n_{\wedge \tau^n_i}\|_\infty\lor\max_{i \in \{0,...,k_n-1\}} \|\Delta x^n_i\|_2)+|\pi_n|)\|\Delta x^n_i\|_2^2,
\end{align*}
\noindent with $s^n_i \in [\tau^n_i,\tau^n_{i+1})$, and
\begin{align*}
r(\epsilon) := d^2\max_{0\leq i,j\leq d}\sup\{|\tilde{\partial}_i \tilde{\partial}_jF(t,y)-\tilde{\partial}_i \tilde{\partial}_jF(s,z)|: \|y_{\wedge t} - z_{\wedge s}\|_\infty + |t-s| < \epsilon\}.
\end{align*}

Since the second-order derivatives of $F$ are continuous, $\lim_{\epsilon \to 0^+} r(\epsilon) = 0$. Now, since $F,\partial F,\tilde{\nabla} F$, and $\tilde{\nabla}^2 F$ are all boundedness preserving, the first two terms in \eqref{eqn:Ito Error} converge to their corresponding term on the right-hand side of \eqref{Ito_along_a_direction}, while the third and fourth terms converge to $0$ since $x^n$ converges to $x$ uniformly, and the fifth term converges to the last integral on the right-hand side of \eqref{Ito_along_a_direction}, since $x$ is of bounded quadratic covariation along the partition $\pi$. Meanwhile, for the error term, if $\epsilon_n = (\|x-x^n\|_\infty\lor\max_{i \in \{0,...,k_n-1\}} \|\Delta x^n_i\|_2)+|\pi_n|$, then
\begin{align*}
    \left\vert\sum_{i = 0}^{k_n-1} R(\tau^n_i,\tau^n_{i+1},x^n_{\wedge \tau^n_{i+1}})\right\vert &\leq r(\epsilon_n)\sum_{i = 0}^{k_n-1} (x(\tau^n_{i+1})-x(\tau^n_i))^2,
\end{align*}
\noindent where the last expression converges to $0$, since $x$ is of bounded quadratic covariation along $\pi$. Therefore, the second term on the right-hand side of \eqref{eqn:Ito Error} converges to an integral $\int_0^T \tilde{\nabla} F(t,x_{\wedge t})\cdot \mathrm{d}^\pi x(t)$ defined by
\begin{align*}
    &\int_0^T \tilde{\nabla} F(t,x_{\wedge t})\cdot \mathrm{d}^\pi x(t) = \lim_{n \to \infty} \sum_{\tau^n_i \in\pi_n} \langle\tilde{\nabla} F(\tau_i^n,x_{\tau^n_i}),\Delta x_i^n\rangle\\
    &= F(T,x_{\wedge T})-F(0,x_{\wedge 0})-\int_0^T\partial F(t,x_{\wedge t})\mathrm{d}t-\frac{1}{2}\int_0^T Tr(\tilde{\nabla}^2 F(t,x_{\wedge t})\mathrm{d}[x]_\pi(t)).
\end{align*}
\end{proof}
If $\mathbb{P}$ is a probability measure such that $X$ is a semimartingale, and $\pi = \{\pi_n\}_{n \in \mathbb{N}}$ is a sequence of partitions such that $|\pi_n| \to 0$ then, as in \cite{Cont13}, $x$ has bounded quadratic covariation along the sequence $\pi$ a.s. Therefore, Theorem~\ref{Coinvariant Ito} introduces the notion of coinvariant derivative (see \cite{Kim}) for functions with pathwise quadratic variation. Moreover, although a similar derivative has been proposed in \cite{Zhang} in the setting of path-dependent differential equations, Theorem~\ref{Coinvariant Ito} characterizes the derivative needed for these problems and shows that the second derivative can be taken to be the coinvariant derivative of the functional's first derivative. This derivative could also be studied as in \cite{Chitasvili, MANIA}; however, there, the proof of the functional It\^{o}'s formula relies heavily on the probability measure, while the above proof is pathwise. 

Since Theorem~\ref{Coinvariant Ito} recovers the functional It\^{o}'s formula in the context of derivatives along paths, classical results derived from this formula also follow. Below, as a sample, we outline the Fisk-Stratonovich formula for functionals of continuous paths, and present a brief application to path-dependent differential equations.\\

Let $X, Y: [0, T] \to \mathbb{R}$ be two continuous semimartingales. For $t \in [0, T]$, recall that the Fisk-Stratonovich integral $\int_0^t X(s) \circ\,\mathrm{d}Y(s)$ is given by,
\begin{align}
\int_0^t X(s)\circ\,\mathrm{d}Y(s) &= \int_0^t X(s)\,\mathrm{d}Y(s) + \frac{1}{2}[X, Y](t).
\end{align}

\noindent 
Then Theorem~\ref{Coinvariant Ito}, applied to the derivative $\tilde \partial_i F$ leads to,

\begin{align*}
\left[\tilde{\partial}_i F(\cdot, X_{\wedge \cdot}), X_i\right](t) &= \sum_{j = 1}^d \int_0^t\tilde{\partial}_i \tilde{\partial_j}F(s, X_{\wedge s})\mathrm{d}[X_i, X_j](s).
\end{align*}

Thus, summing over $i$ gives the following functional Fisk-Stratonovich formula:  Let $F \in \mathfrak{C}^{3}([0,T],\mathbb{R}^d)$, and let $(X(t))_{t \in [0,T]}$ be a.s.~a continuous semimartingale.  Then, $\mathbb{P}$-a.s.~, 
\begin{align*}
    F(t,X_{\wedge t}) - F(0,X_{\wedge 0}) = \int_0^t \partial F(s, X_{\wedge s})ds +  \sum_{i = 1}^d \int_0^t \tilde{\partial}_i F(s, X_{\wedge s}) \circ \mathrm{d}X_i(s).
\end{align*}

To finish, we illustrate another use of functional derivatives by obtaining, in our setting, a very classical result, namely, the Feynman-Kac formula. This version of the formula characterizes the differential behavior of a functional given by the conditional expectation with respect to a path's history, even when the functionals involved in its definition cannot be extended to include c\`{a}dl\`{a}g functions on their domain. Its proof follows the same outline as that presented through horizontal derivatives in \cite[Theorem 2]{Dupire09}.

\begin{proposition}
Let $X:[0, T] \to \mathbb{R}^d$ be a semimartingale satisfying
\begin{align*}
dX(s) = a(s, X_{\wedge s})\mathrm{d}s + \sigma(s, X_{\wedge s^-})\mathrm{d}B(s),
\end{align*} 
with $a = (a_1,...,a_d)$, and $\sigma = (\sigma_{i,j})_{1\leq i\leq d, 1\leq j\leq m}$ such that $a_i \in \mathscr{C}^{0,0}([0,T],\mathbb{R}^d)$, $\sigma_{i,j} \in \mathscr{C}^{0,0}([0,T],\mathbb{R}^d)$, and $B$ a standard $m$-dimensional Brownian motion. Furthermore, let $g \in \mathscr{C}^{0,0}([0,T],\mathbb{R}^d), r \in \mathscr{C}^{0,0}([0,T],\mathbb{R}^d)$ be such that $g$ is integrable and $r$ is positive. Define the non-anticipative functional $f \in \mathscr{C}^{0,0}([0,T],\mathbb{R}^d)$ by,
\begin{align*}
f(t, x) &= \mathbb{E}\left[e^{-\int_t^T r(s, X_{\wedge s})\,\mathrm{d}s} g(T,X_{\wedge T})|X_{\wedge t} = x_{\wedge t}\right].
\end{align*}
Then, if $f \in \mathfrak{C}^2([0, T], \mathbb{R}^d)$, its functional derivatives satisfy 
\begin{align}\label{F-Kac}
\partial f(t, X_{\wedge t}) + \langle a(t, X_{\wedge t}), \tilde{\nabla}f(t, X_{\wedge t^-})\rangle - r(t, X_{\wedge t})f(t,X_{\wedge t})\nonumber &\\
+ \frac{1}{2}Tr(\tilde{\nabla}^2 f(t, X_{\wedge t}) \sigma(t, X_{\wedge t})\sigma(t, X_{\wedge t})^T) = 0&\,\,\text{a.s.}
\end{align} 
\begin{proof}
Let $h \in \mathscr{C}^{0,0}([0,T],\mathbb{R}^d)$ satisfy
\begin{align*}
h(t, X_{\wedge t}) = e^{-\int_0^t r(s, X_{\wedge s})\,\mathrm{d}s}f(t,X_{\wedge t}).
\end{align*}
By the tower property of conditional expectation,
\begin{align*}
h(t, X_{\wedge t}) = \mathbb{E}\!\left[e^{-\int_0^T r(s, X_{\wedge s})\,\mathrm{d}s}\,
g(T,X_{\wedge T})\,\Big|\,\mathcal{F}_t\right].
\end{align*}
Since $r$ is positive and $g$ is integrable, the terminal variable
$e^{-\int_0^T r\,\mathrm{d}s}\,g(T,X_{\wedge T})$ is integrable, so $h$
is a true martingale. On the other hand, a direct application of
Theorem~\ref{Coinvariant Ito} gives
\begin{align}\label{h-dynamics}
dh(t, X_{\wedge t}) = e^{-\int_0^t r(s, X_{\wedge s})\,\mathrm{d}s}
\Big( -r(t, X_{\wedge t})f(t,X_{\wedge t})+\partial f(t, X_{\wedge t})
+ \langle a(t, X_{\wedge t}),\tilde{\nabla}f(t, X_{\wedge t^-})\rangle
\nonumber &\\
+ \tilde{\nabla}f(t, X_{\wedge t^-})^T\sigma(t, X_{\wedge t^-})dB(t)
+ \frac{1}{2}Tr(\tilde{\nabla}^2 f(t, X_{\wedge t})
\sigma(t, X_{\wedge t})\sigma(t, X_{\wedge t})^T\Big)
\,\,\text{a.s.}&
\end{align}
By the uniqueness of the semimartingale decomposition, the drift terms
in \eqref{h-dynamics} must vanish. Canceling the exponential factor,
we arrive at \eqref{F-Kac}.
\end{proof}
\end{proposition}

\subsection*{Declaration of Generative AI and AI-assisted technologies in the writing process}

During the preparation of this work the authors used Claude (Anthropic) in order to assist with structural review and editing of the manuscript, to explore different approaches for the construction of the counterexample in Section~\ref{sec:counterexample}, and to draft and revise selected passages of expository text. After using this tool, the authors reviewed and edited all content as needed and take full responsibility for the content of the publication.

%%%%%%%%%%%%%%%%%%%%%%%%%%%%%%%%%%%%%%%%%%%%%%%%%%%%%%%%%%%%%%%%%%%
%%                                                               %%
%% Supplementary Material, if any, should be provided in         %%
%% {supplement} environment  with title and short description.   %%
%%                                                               %%
%%%%%%%%%%%%%%%%%%%%%%%%%%%%%%%%%%%%%%%%%%%%%%%%%%%%%%%%%%%%%%%%%%%

%%%%%%%%%%%%%%%%%%%%%%%%%%%%%%%%%%%%%%%%%%%%%%%%%%%%%%%%%%%%%%%%%%%
%%                                                               %%
%% Use the two commands below for producing your bibliography    %%
%% with bibtex, then comment again the commands and include the  %%
%% content of the .bbl file in this file below the commands.     %%
%%                                                               %%
%%%%%%%%%%%%%%%%%%%%%%%%%%%%%%%%%%%%%%%%%%%%%%%%%%%%%%%%%%%%%%%%%%%

%\bibliographystyle{amsplain}
%\bibliography{Bibliography}

% add below the content of your .bbl file produced by bibtex.
\providecommand{\bysame}{\leavevmode\hbox to3em{\hrulefill}\thinspace}
\providecommand{\MR}{\relax\ifhmode\unskip\space\fi MR }
% \MRhref is called by the amsart/book/proc definition of \MR.
\providecommand{\MRhref}[2]{%
  \href{http://www.ams.org/mathscinet-getitem?mr=#1}{#2}
}
\providecommand{\href}[2]{#2}

%%%%%%%%%%%%%%%%%%%%%%%%%%%%%%%%%%%%%%%%%%%%%%%%%%%%%%%%%%%%%%%%%%%
%%                                                               %%
%% You may add acknowledgments (optional).                       %%
%%                                                               %%
%%%%%%%%%%%%%%%%%%%%%%%%%%%%%%%%%%%%%%%%%%%%%%%%%%%%%%%%%%%%%%%%%%%

%%%%%%%%%%%%%%%%%%%%%%%%%%%%%%%%%%%%%%%%%%%%%%%%%%%%%%%%%%%%%%%%%%%
%%                                                               %%
%% You have reached the end of your document.                    %%
%%                                                               %%
%%%%%%%%%%%%%%%%%%%%%%%%%%%%%%%%%%%%%%%%%%%%%%%%%%%%%%%%%%%%%%%%%%%

\end{document}